\documentclass{amsart} 

\usepackage{amsmath,amssymb,amsthm} 
\usepackage{graphicx} 
\usepackage{subfig} 
\usepackage[arrow, matrix, curve]{xy} 
\usepackage{color} 
\usepackage{float} 
\usepackage{changepage}

\usepackage{makecell} 
\usepackage{multirow} 
\usepackage{hyperref} 
\usepackage{enumerate}
\usepackage{booktabs} 
\usepackage{siunitx} 

\newcommand{\R}{\mathbb{R}}
\newcommand{\Z}{\mathbb{Z}}

\newcommand{\N}{\mathbb{N}}
\newcommand{\C}{\mathbb{C}}

\newtheoremstyle{thm}{}{}{\itshape}{}{\bfseries}{}{ }{} 
\newtheoremstyle{definition}{}{}{}{}{\bfseries}{}{ }{} 

\theoremstyle{thm}
\newtheorem{Theorem}{Theorem}[section]
\newtheorem{thm}[Theorem]{Theorem}
\newtheorem{lem}[Theorem]{Lemma}
\newtheorem{prop}[Theorem]{Proposition}
\newtheorem{cor}[Theorem]{Corollary}
\newtheorem*{Theorem-ohne}{Theorem}

\theoremstyle{definition}

\newtheorem{rem}[Theorem]{Remark}
\newtheorem{ex}[Theorem]{Example}

\begingroup\expandafter\expandafter\expandafter\endgroup
\expandafter\ifx\csname pdfsuppresswarningpagegroup\endcsname\relax
\else
\fi

\definecolor{amaranth}{rgb}{0.9, 0.17, 0.31} 
\definecolor{carrotorange}{rgb}{0.93, 0.57, 0.13} 
\definecolor{citrine}{rgb}{0.89, 0.82, 0.04} 
\definecolor{dartmouthgreen}{rgb}{0.05, 0.5, 0.06} 
\definecolor{ballblue}{rgb}{0.13, 0.67, 0.8} 
\definecolor{ceruleanblue}{rgb}{0.16, 0.32, 0.75} 
\definecolor{amethyst}{rgb}{0.6, 0.4, 0.8} 
\definecolor{amber}{rgb}{1.0, 0.75, 0.0} 
\definecolor{burlywood}{rgb}{0.87, 0.72, 0.53} 

\numberwithin{equation}{section}

\begin{document}


\title{The classification of surfaces via normal curves} 

\author{Fethi Ayaz}
\address{Humboldt-Universit\"at zu Berlin, Rudower Chaussee 25, 12489 Berlin, Germany.}
\email{fethi.oemer.ayaz@cms.hu-berlin.de, fethi.o.ayaz@hotmail.com}

\author{Marc Kegel}
\address{Humboldt-Universit\"at zu Berlin, Rudower Chaussee 25, 12489 Berlin, Germany.}
\email{kegemarc@math.hu-berlin.de, kegelmarc87@gmail.com}

\author{Klaus Mohnke}
\address{Humboldt-Universit\"at zu Berlin, Rudower Chaussee 25, 12489 Berlin, Germany.}
\email{mohnke@mathematik.hu-berlin.de}


\date{\today} 

\begin{abstract}
We present a simple proof of the surface classification theorem using normal curves. This proof is analogous to Kneser's and Milnor's proof of the existence and uniqueness of the prime decomposition of $3$-manifolds. In particular, we do not need any invariants from algebraic topology to distinguish surfaces.
\end{abstract}

\keywords{Classification of surfaces, normal curves, prime decomposition of manifolds}

\makeatletter
\@namedef{subjclassname@2020}{%
  \textup{2020} Mathematics Subject Classification}
\makeatother

\subjclass[2020]{57M25; 57R65, 57M12} 

\maketitle


\section{Introduction}

One of the central problems in topology asks for the classification of manifolds in a fixed dimension, where often the simplifying assumption is imposed that the manifolds are connected, orientable and closed (i.e. compact and without boundary). While it is not hard to show that any $1$-manifold is homeomorphic to the $1$-dimensional circle $S^1$ (we refer to~\cite[Section~7.1]{Fr} for a detailed proof, a different proof for smooth $1$-manifolds is given in~\cite{Mi97}) the first non-trivial classification result is the classification of surfaces. Throughout this paper, we will use the word surface to describe a closed, connected, orientable, triangulated, and smooth $2$-dimensional manifold.

\begin{thm}\label{thm:class}
For any surface $F$ there exist a unique non-negative integer $g$, the \textbf{genus} of $F$, such that $F$ is diffeomorphic to $\#_g T^2$, the $g$-fold connected sum of $2$-tori. 
\end{thm}

Where in the above theorem $\#_0 T^2$ is defined to be the $2$-sphere $S^2$ and $\#_1 T^2$ just denotes a $2$-torus $T^2=S^1\times S^1$. Theorem~\ref{thm:class} is often attributed to Rad\'o~\cite{Ra26}, who proved that any topological $2$-manifold can be triangulated, cf.~\cite{Mo77}. Given a triangulated surface $F$ it is not difficult to actually construct a smooth structure on $F$. For a direct proof (not using any triangulation) that any surface admits a smooth structure, we refer to~\cite{Ha13} where the smooth structure is constructed using Kirby's torus trick which is one of the main ingredients to study the existence of PL structures on higher dimensional manifolds~\cite{KS77}. Moreover, it is known that every homeomorphism between smooth surfaces is isotopic to a diffeomorphism, see~\cite{Ha13} for a modern proof. It follows that Theorem~\ref{thm:class} remains also true for topological surfaces.

In the classical proof of Theorem~\ref{thm:class} one chooses a triangulation of a given surface $F$ and modifies this triangulation without changing the diffeomorphism type of the underlying surface by cut and paste topology into a standard form from which one can read-off that $F$ is diffeomorphic to a connected sum of tori, see for example the classical~\cite{ST34}.

A modern, more elegant, approach is given by using handle decompositions (as done in~\cite{Ge17} or~\cite[Chapter~108]{Fr}) which need the surface to admit a smooth structure. A similar approach is taken in Conway's ZIP proof~\cite{FW99}. 

To conclude the proof in both approaches, in the classical and the modern approach, one needs to argue that $\#_g T^2$ and $\#_{g'} T^2$ are not homeomorphic if $g\neq g'$. This can be achieved by computing invariants from algebraic topology such as the Euler characteristic, the (abelization of the) fundamental group or the homology groups.

In this note, we will discuss a different proof of Theorem~\ref{thm:class} using normal curves which is analogous to Kneser's and Milnor's proof of the existence and uniqueness of the prime decomposition of $3$-manifolds~\cite{Kn29,Mi62}. For more details on the prime decompositions of $3$-manifolds, we also refer to Chapter~3 of~\cite{He76}. In particular, our proof will not use any invariants from the machinery of algebraic topology. On the other hand, we have to appeal to the Schoenflies theorem~\cite{Sc06}, which can be shown by elementary methods (see for example Chapter~3 of~\cite{Mo77}). 

In Section~\ref{sec:normal_curves} we will briefly introduce the necessary background on normal curves. Using normal curves we will show in Section~\ref{sec:prime_surfaces} that $T^2$ is a prime surface (i.e. it cannot be decomposed non-trivially as a connected sum) and that the only other prime surface is $S^2$. To conclude the proof of Theorem~\ref{thm:class} we will show in Sections~\ref{sec:existence} and~\ref{sec:uniqueness} that any surface admits a unique prime decomposition. Finally, in Section~\ref{sec:higher_dimensions} we will briefly relate our approach to results for higher dimensional manifolds.

\subsection*{Conventions:} We work in the smooth category. All manifolds, maps, and ancillary objects are assumed to be smooth. All arguments would also work in the PL category. 

\section{Normal curves on triangulated surfaces}\label{sec:normal_curves}

The idea to prove Theorem~\ref{thm:class} is that the collection of curves on a given surface $F$ will determine $F$. To actually make the collection of curves into an object that we can handle effectively we will restrict to a simple subclass of curves that contain still all interesting curves but admit a nice countable structure. These are the so-called normal curves.

In this section we will introduce the needed concepts about normal curves. More on normal curves can be found in Section~3.2 of~\cite{Ma07}, another source is~\cite{Sc14}. 

First, we recall the concept of a singular triangulation of a surface.
We denote by $\Delta$ a $2$-dimensional triangle seen as a subspace of $\R^2$. A (singular) \textbf{simplicial complex} is obtained by taking finitely many copies of $\Delta$ and identifying all their edges in pairs via affine homeomorphisms. A (singular) \textbf{triangulation} $T$ of a surface $F$ is a homeomorphism of $F$ to a singular simplicial complex. 

\begin{rem}
 In contrast to a genuine triangulation a singular triangulation needs much fewer triangles to build a given surface. Since the theory of normal curves works equally well for singular triangulations, we choose to work with singular triangulations here.
\end{rem}

\begin{figure}[htbp] 
	\centering
	\def\svgwidth{0,5\columnwidth}
	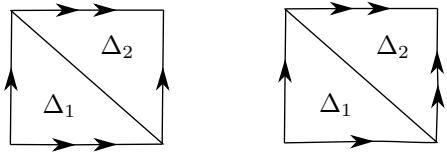
	\caption{Singular triangulations of $T^2$ (left) and $S^2$ (right) consisting of two triangles each. On the other hand, it is not hard to show that the minimal number of triangles needed to genuine triangulate $S^2$, respectively $T^2$, is $4$, respectively $14$.}
	\label{fig:singular_triangulations}
\end{figure}

We denote by $F$ always a surface with a fixed choice of a triangulation. We denote the triangles in the triangulation of $F$ by
\begin{equation*}
    \{\Delta_1,\ldots,\Delta_n\}.
\end{equation*}

A closed $1$-dimensional submanifold $c$ of $F$ is called \textbf{curve system} and if $c$ is connected just \textbf{curve}. By general position, we can assume after a small perturbation of $c$ that $c$ intersects the edges $E$ of the triangles transversely and does not intersect the vertices $V$. Then it follows from the Schoenflies theorem~\cite{Sc06} that any component of $c\cap \Delta_i$ is either a closed curve bounding a disk in the interior of $\Delta_i$, an embedded arc that connects two points on the same edge of $\Delta_i$ or an embedded arc connecting two different edges of $\Delta_i$, as depicted in Figure~\ref{fig:normal_curves}.

\begin{figure}[htbp] 
	\centering
	\def\svgwidth{0,85\columnwidth}
	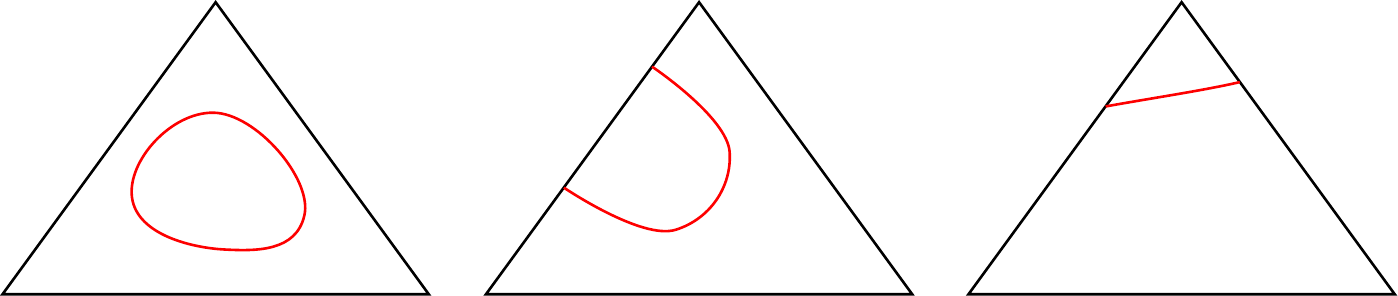
	\caption{Possible components of the intersection $c\cap \Delta$ of a curve system $c$ with a triangle $\Delta$. If only the right configuration (or its rotations) occurs, $c$ is called normal.}
	\label{fig:normal_curves}
\end{figure}

A curve system $c$ on $F$ is called \textbf{normal} if for any triangle $\Delta_i$ of the triangulation every component of $c\cap \Delta_i$ is an arc connecting two different edges of $\Delta_i$.

Actually, it is not a real restriction to restrict to normal curves as the next result shows.

\begin{thm}\label{thm:normalization}
Any curve system c on a surface $F$ can be deformed into a normal curve system by isotopy and possibly removing components bounding disks in $F$.
\end{thm}

\begin{proof}
Let $c$ be a curve system on $F$. We will describe an algorithm, called \textbf{normalization}, that deforms $c$ into a normal curve system.

To show that our algorithm terminates in finite time, we define the complexity $d(c)$ of a curve system $c$ as the sum of the number of components of $c$ and the intersection points of $c$ with the edges of the triangulation. We emphasize the obvious fact that $d$ is bounded from below by $0$.

\noindent \textbf{Step 1:} If one component of $c$ bounds a disk inside a triangle $\Delta$ as in Figure~\ref{fig:normal_curves} (left), we remove that component. This will decrease the complexity $d(c)$ by one and thus after finitely many executions of Step~$1$, we end up with a curve system, $c'$ without configurations as on the left of Figure~\ref{fig:normal_curves}. 

The curve system $c'$ is not isotopic to $c$ but only differs from $c$ by finitely many curves that bound disks in $F$. This is the only place where the curves are not changed by isotopy. For simplicity we denote $c'$ again by $c$.

\noindent \textbf{Step 2:} If there exists a triangle $\Delta$ such that one component of $c\cap\Delta$ is an arc that connects two points on the same edge of $\Delta$ as in the middle of Figure~\ref{fig:normal_curves} we take the innermost of those components and isotope it into the nearby triangle, as shown in Figure~\ref{fig:normalization}. (Here we use again the Schoenflies theorem.) This isotopy will reduce a pair of intersection points with an edge and thus reduce the complexity $d(c)$ by two. Then we continue again with Step~$1$. 

\begin{figure}[htbp] 
	\centering
	\def\svgwidth{0,99\columnwidth}
	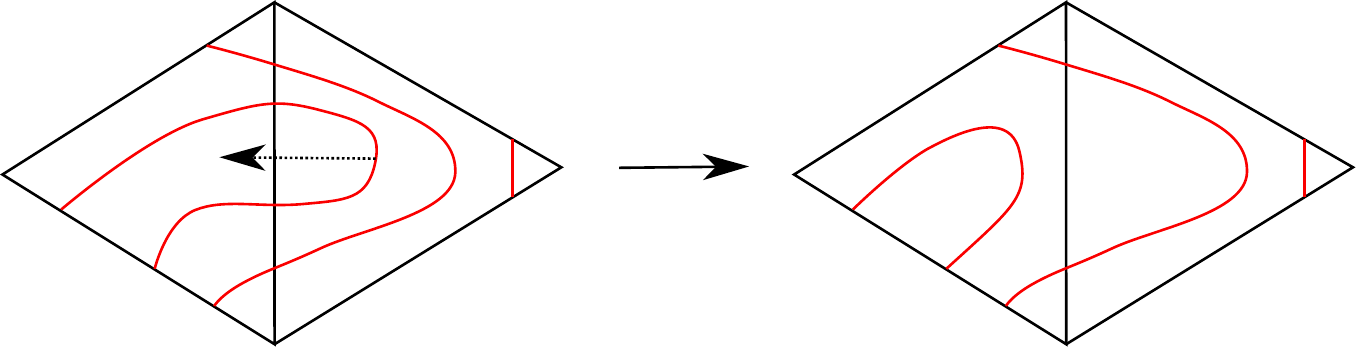
	\caption{Step $2$ in the normalization process removes two intersection points of $c$ with an edge of the triangulation.}
	\label{fig:normalization}
\end{figure}

Since in both steps the complexity is reduced, the algorithm will terminate after finitely many executions and end with a curve system $c$ that is normal.
\end{proof}

We remark that a curve (if it bounds a disk) might be deformed to the empty set by the above algorithm. For an \textbf{essential} curve (i.e. one that does not bound a disk) the above deformation is always an isotopy and in particular, we can always choose a representative that is normal.

One main feature of normal curves is that we can describe them uniquely (up to isotopy through normal curves) by non-negative integer solutions of simple linear equations. For that we enumerate the angles in all triangles of the triangulation by $1,\ldots, 3n$ and assign to a given normal curve system $c$ an integral vector $x(c)\in\N_0^{3n}$ with $i$-th entry $x_i$ the number of arcs in $c$ that run through the $i$-th angle.

We observe that for any edge $e$ of the triangulation there are exactly $4$ adjacent angles. Let $x_i$ and $x_j$ the coordinates belonging to the angles on one side of $e$ and $x_k$ and $x_l$ the coordinates belonging to the other side of $e$. Then we observe that $x_i+x_j=x_k+x_l$, since the number of components of $c$ on both sides of $e$ has to be equal. This yields the so-called \textbf{matching system}
\begin{align*}
    &x_i+x_j=x_k+x_l,  \text{ for all } x_i,x_j,x_k,x_l \text{ belonging to the same edge},\\
    &x_i\geq0, \text{ for } 1\leq i\leq 3n.
\end{align*}

\begin{thm}\label{thm:matchingSystem}
The map 
\begin{align*}
    \big\{\text{normal curves}\big\}/_\sim&\longrightarrow \big\{\text{solutions of the matching system}\big\}\\
    c&\longmapsto x(c)
\end{align*}
is a bijection from the set of normal curves (up to normal isotopy) to the solutions of the matching system.
\end{thm}

\begin{proof}
To construct the inverse we assign to any solution of the matching system a normal curve as follows. For any angle $i$ we draw $x_i$ parallel arcs that join the sides of the angle. If $x$ is a solution of the matching system, the number of arcs coming from one side to an edge equals the number of arcs coming from the other side. Thus there is a unique way to connect these arcs to a normal curve system $c$ belonging to the solution $x$.
\end{proof}

\begin{ex}\label{ex:matchingsystemT2}
We consider the triangulation of $T^2$ shown in Figure~\ref{fig:T2} and fix the labeling of the angles as depicted in that figure. 
\begin{figure}[htbp] 
	\centering
	\def\svgwidth{0,85\columnwidth}
	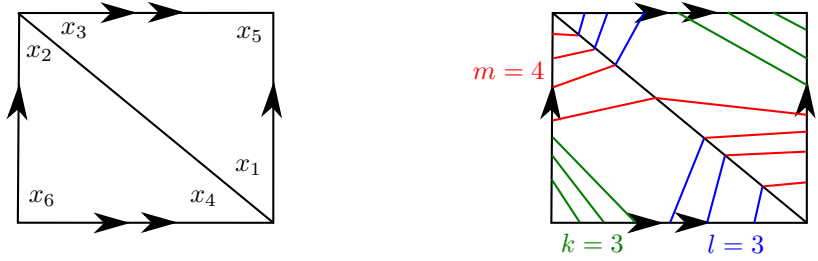
	\caption{Left: A singular triangulation of $T^2$ and a labeling of its angles. Right: An example of a normal curve system.}
	\label{fig:T2}
\end{figure}
Then the matching system is given by
\begin{align*}
    x_1+x_3=x_2+x_4,\\
    x_1+x_5=x_2+x_6,\\
    x_3+x_5=x_4+x_6,\\
    x_1,\ldots, x_6\in\N_0.
\end{align*}
We solve the matching system and set $m:=x_1=x_2$, $l:=x_3=x_4$, and $k:=x_5=x_6$. Thus we see that the set of normal curves $\mathcal N$ on $T^2$ is given by
\begin{align*}
   \mathcal N=\big\{ C_{m,l,k}=(m,m,l,l,k,k)\,\big\vert\, m,l,k\geq0 \big\}.
\end{align*}
By Theorem~\ref{thm:normalization} every essential curve on $T^2$ is contained in $\mathcal N$.
\end{ex}

\section{Classification of prime surfaces}\label{sec:prime_surfaces}

In this section, we will classify all prime surfaces. More concretely we will show that any prime surface is diffeomorphic to $S^2$ or $T^2$. We will start by showing that these two surfaces are prime.

We recall that a surface $F$ is \textbf{prime} if there is no description of $F$ as $F_1\#F_2$ for surfaces $F_1$ and $F_2$ that are both not diffeomorphic to $S^2$. And thus a surface is prime if and only if it does not contain a \textbf{reducing} curve, i.e.\ an essential and separating curve.

\begin{lem}
$S^2$ is prime.
\end{lem}

\begin{proof}
The Schoenflies theorem~\cite{Sc06} implies that any curve on $S^2$ bounds a disk and thus $S^2$ does not contain any essential curve. 
\end{proof}

\begin{lem}\label{lem:T2prime}
$T^2$ is prime.
\end{lem}

\begin{proof}
We will show that any essential curve on the $2$-torus is non-separating. For that, we use the triangulation and the matching system $\mathcal N$ from Example~\ref{ex:matchingsystemT2}. By Theorem~\ref{thm:normalization} any essential curve is isotopic to a normal curve and Theorem~\ref{thm:matchingSystem} implies that any essential curve is given by a solution of the matching system $\mathcal N$.

First, we observe in Figure~\ref{fig:T22} that any normal curve $C_{m,l,k}$ with $m,l,k\geq1$ contains a curve that bounds a disk and is thus non-essential. So it is enough to consider only normal curves with at least one of $m,l$, or $k$ vanishing.

\begin{figure}[htbp] 
	\centering
	\def\svgwidth{0,85\columnwidth}
	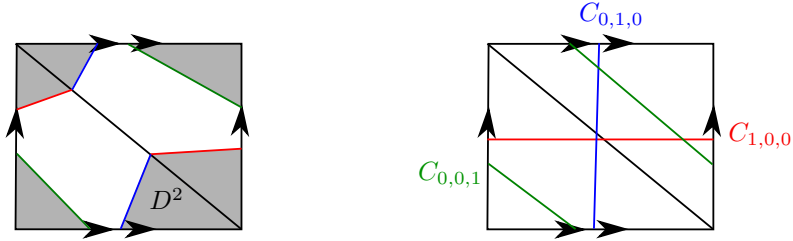
	\caption{Left: Every normal curve system with $m,l,k\geq1$ contains a curve that bounds a disk. Right: 3 simple solutions of the matching system that represent essential curves on $T^2$.}
	\label{fig:T22}
\end{figure}

We choose the standard meridian-longitude pair $(\mu,\lambda)$ on $T^2$ inducing the standard orientation on $T^2$ such that $\mu=C_{1,0,0}$ and $\lambda=C_{0,1,0}$ (as unoriented curves), see Figure~\ref{fig:T22}. Then it follows that $C_{0,0,1}=\mu-\lambda$. Therefore, we have
\begin{align*}
    C_{m,l,0}&=m\mu+l\lambda,\\
    C_{m,0,k}&=(m+k)\mu-k\lambda,\\
    C_{0,l,k}&=k\mu+(l-k)\lambda.
\end{align*}
We remark here that $C_{m,l,0}=m\mu+l\lambda$ and $C_{0,l+m,m}=m\mu+l\lambda$ are isotopic (although not isotopic as normal curves). Similarly $C_{m,0,k}=(m+k)\mu-k\lambda$ and $C_{0,m,m+k}=(m+k)\mu-k\lambda$ are isotopic.

It follows that any essential curve on $T^2$ can be written (up to isotopy) as $p\mu+q\lambda$ with $p\in\N_0$ and $q\in\Z$. Finally, it is not hard to see that a normal curve system $p\mu+q\lambda$ is connected if and only if $p$ and $q$ are coprime and in that case the complement is connected. We have shown that $T^2$ does not contain any essential separating curve.
\end{proof}

\begin{rem}
That any essential curve on a $2$-torus is isotopic to a curve of the form $p\mu+q\lambda$ for $p$ and $q$ coprime is already a non-trivial result, see for example~\cite[Chapter~2]{Ro76} for a different proof.
\end{rem}

Next, we will show that these two examples are actually the only prime surfaces.

\begin{prop}\label{prop:prime_surfaces}
Let $F$ be a prime surface, then $F$ is diffeomorphic to $S^2$ or $T^2$.
\end{prop}

\begin{proof}
First, we show that any triangulated surface $F$ admits a polygonal description, i.e. $F$ can be obtained from a triangulated convex polygon $P$ in the plane by identifying pairs of boundary edges of the polygon. We construct $P$ inductively. We start with an arbitrary triangle $\Delta_1$ in the triangulation of $F$ and set $P_1=\Delta_1$. Next, we glue to one of its edges a different adjacent triangle $\Delta_2$ to obtain the polygon $P_2$, which we can deform to be convex. We continue inductively by gluing an adjacent triangle $\Delta_i$ to the polygon $P_{i-1}$ along one of its edges to get the polygon $P_i$ which we again deform to be convex. Here we always choose a triangle $\Delta_i$ that is not yet contained in $P_{i-1}$. It is straightforward to show that after finitely many steps we are left with a triangulated polygon $P$ in the plane that contains all triangles in the triangulation. From $P$ we can obtain $F$ by identifying pairs of boundary edges.

Now we assume that $F$ is prime. We take two edges $e_1$ and $e_2$ of the triangulated polygon that get identified and choose an embedded normal arc $c$ in the polygon connecting $e_1$ and $e_2$ that induces a normal curve on $F$, see Figure~\ref{fig:prime}. Such a normal arc $c$ exists by performing the normalization process from the proof of Theorem~\ref{thm:normalization} in this setting. 

\begin{figure}[htbp] 
	\centering
	\def\svgwidth{0,85\columnwidth}
	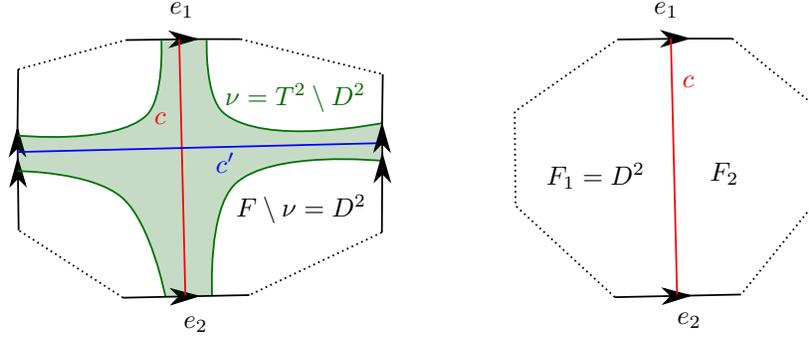
	\caption{Left: If a prime surface $F$ contains a curve $c$ whose complement is connected it is diffeomorphic to $T^2$. Right: If the complement of $c$ is disconnected we can simplify its polygonal description by collapsing $F_1=D^2$.}
	\label{fig:prime}
\end{figure}

First, we discuss the case that $F\setminus c$ is path-connected. Then we can find a normal curve $c'$ intersecting $c$ transversely in a single point, see Figure~\ref{fig:prime} (left). Since $F$ is orientable a regular neighborhood $\nu$ of $c\cup c'$ is diffeomorphic to $T^2\setminus D^2$. Now we observe that the boundary $\partial \nu$ of $\nu$ is a curve in $F$ and since $F$ is prime $\partial\nu$ has to bound a disk in $F$ and we conclude that $F$ is diffeomorphic to $T^2$.

In the other case, when $F\setminus c$ consists of two components $F_1$ and $F_2$, we will argue that we can simplify the triangulated polygon. Indeed, since $F$ is prime, we know that $c$ bounds a disk in $F$, and thus one of the components of its complement, say $F_1$, has to be a disk $D^2$. It follows that collapsing $F_1$ to a point will not change the topology of $F$. Thus we get a polygon $P'$ with fewer edges also describing $F$ as follows. We cut $P$ along $c$, remove the part of $P$ that belongs to $F_1$, and collapse the part of the boundary belonging to $c$, $e_1$ and $e_2$, see Figure~\ref{fig:prime} (right).

After finitely many of those steps we either find a normal curve with connected complement and deduce that $F$ is diffeomorphic to a $2$-torus or we reduce the polygon to the bigon polygon describing $S^2$.
\end{proof}

\section{The existence of the prime decomposition for surfaces}\label{sec:existence}

In this section, we will show that any surface can be decomposed into prime surfaces. This implies that any surface is diffeomorphic to the connected sum of some number of $2$-tori. 

\begin{thm}\label{thm:existence}
Any surface $F$ can be decomposed into prime factors, i.e. $F$ is diffeomorphic to $F_1\#\ldots\#F_g$ with $F_i$ prime surfaces.
\end{thm}

\begin{proof}
The proof follows the naive approach of repeatedly cutting the surface along separating essential curves until only prime pieces remain. We need to argue why this process stops after finite time. (In fact, this process will, in general, not stop in higher dimensions, see Section~\ref{sec:higher_dimensions}.)

For that, we will show that there exists for every surface $F$ a natural number $g$, such that for any curve system $c$ consisting of $g$ disjoint separating essential curves  $c_1,\ldots c_g$ at least two curves $c_i$ and $c_j$ are isotopic.

We choose a triangulation $\{\Delta_1,\ldots,\Delta_n\}$ of $F$. By Theorem~\ref{thm:normalization} we can assume that $c$ is normal. We observe that for any triangle $\Delta_i$ in the triangulation the complement $\Delta_i\setminus c$ of $c$ consists of quadrilaterals $Q_j$ and at most four other exceptional pieces $E_j$, see Figure~\ref{fig:existence} (left). 
\begin{figure}[htbp] 
	\centering
	\def\svgwidth{0,99\columnwidth}
	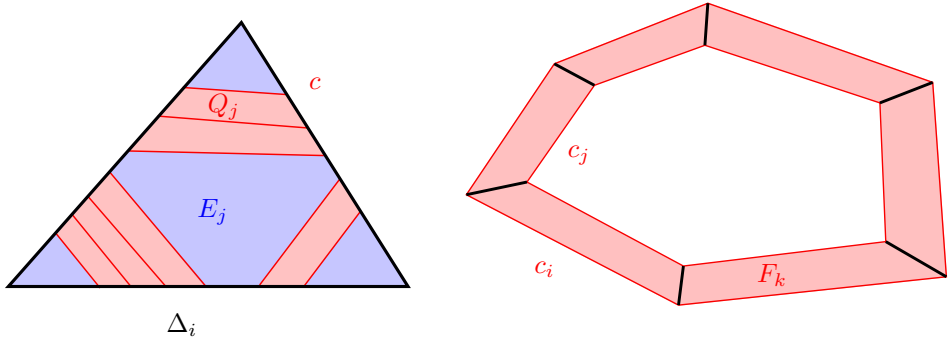
	\caption{Left: $\Delta_i\setminus c$ consists of quadrilaterals $Q_j$ (in red) and at most four exceptional pieces $E_j$ (in blue). Right: The component $F_k$ is obtained by gluing together only quadrilaterals and thus forms an annulus between two components $c_i$ and $c_j$.}
	\label{fig:existence}
\end{figure}
Every quadrilateral has the property that two opposite edges are subsets of $c$ and the other two edges are subsets of the edges of the triangulation. Since the triangulation consists of $n$ triangles $\Delta_i$ we see that the complement of $c$ decomposes as
\begin{equation*}
    F\setminus c=\bigcup\limits_{j=1}^{m} Q_j \cup \bigcup\limits_{j=1}^{l} E_j,
\end{equation*}
where $l$, the total number of exceptional pieces, is at most $4n$ and thus bounded by a quantity that only depends on the triangulation of the surface $F$ and not on $c$. On the other hand, we know that all $c_i$ are separating and thus $F\setminus c$ decomposes as disjoint union of $g+1$ pieces, i.e.
\begin{equation*}
    F\setminus c= F_1\sqcup \cdots\sqcup F_{g+1}.
\end{equation*}
If we now choose $g\geq4n\geq l$, it follows that at least one $F_k$ does not contain any exceptional pieces $E_j$ and thus $F_k$ is obtained by gluing together quadrilaterals $Q_j$, see Figure~\ref{fig:existence} (right). Since the surface $F$ is orientable, we conclude that $F_k$ is an annulus between two curves $c_i$ and $c_j$ and thus $c_i$ is isotopic to $c_j$. 
\end{proof}

As a direct corollary, we obtain the following.

\begin{cor}
For any surface $F$ there exists a non-negative integer $g$ such that $F$ is diffeomorphic to $\#_g T^2$.
\end{cor}

\begin{proof}
From Theorem~\ref{thm:existence} we know that $F$ is diffeomorphic to $F_1\#\ldots\#F_{g'}$ with $F_i$ prime surfaces. But from Proposition~\ref{prop:prime_surfaces} we know that any $F_i$ is diffeomorphic to a $2$-torus or a $2$-sphere.
\end{proof}

\section{The uniqueness of the prime decomposition for surfaces}\label{sec:uniqueness}

To finish the proof of Theorem~\ref{thm:class} it remains to argue that the genus of a surface is actually an invariant of its diffeomorphism type. In this section, we will achieve this by showing that the prime decomposition of a surface is unique (up to adding $S^2$-factors).

\begin{thm} \label{thm:unique}
If $\#_g T^2$ is diffeomorphic to $\#_{g'} T^2$ then $g$ and $g'$ are equal.
\end{thm}

For the proof of Theorem~\ref{thm:unique} we will need the following two well-known observations about diffeomorphisms of surfaces.

\begin{lem}[Alexander trick]\label{lem:Alex}
    Any diffeomorphism of $S^1=\partial D^2$ extends to a diffeomorphism of $D^2$.
\end{lem}

\begin{proof}
Let $f\colon S^1\rightarrow S^1$ be a diffeomorphism. After possibly composing with a reflection that extends over $D^2$ we can assume that $f$ is orientation preserving. We will show that $f$ is isotopic to the identity. First, we observe that after an isotopy we can assume that $f$ preserves $0\in S^1=\R/{2\pi \Z}$. We write $\tilde f$ for the lift of $f$ to a map $\R\rightarrow \R$. Then we observe that the linear interpolation between $\tilde f$ and the identity is equivariant and thus induces an isotopy between $f$ and the identity. By performing this isotopy on a tubular neighborhood of $S^1$ in $D^2$ we deduce that $f$ extends over $D^2$.
\end{proof}

\begin{lem}\label{lem:Dehn}
    Let $c$ and $c'$ be non-separating curves on a surface $F$. Then there exists a diffeomorphism of $F$ that sends $c$ to $c'$.
\end{lem}

\begin{proof}
It is not hard to explicitly construct such a diffeomorphism as a composition of Dehn twists. Instead of repeating this standard argument here we refer for example to Lemma 13.3 in~\cite{PS97} for a detailed proof.
\end{proof}

\begin{proof}[Proof of Theorem~\ref{thm:unique}]
Let $F$ be a surface that is diffeomorphic to two connected sum decompositions, to $F_1\#\cdots\# F_g$ and to $F'_1\#\cdots\# F'_{g'}$, with $F_1,\ldots,F_g$ and $F'_1,\ldots,F'_{g'}$ prime surfaces that are not $2$-spheres.
    From Proposition~\ref{prop:prime_surfaces} it follows that any $F_i$ and any $F'_i$ is diffeomorphic to a $2$-torus. From the Schoenflies theorem, it follows that any curve on $S^2$ separates $S^2$. On the other hand, for any $g\geq1$ there exists a curve on $\#_g T^2$ that does not separate. Thus $g=0$ if and only if $g'=0$ and we assume without loss of generality that $g,g'>0$.

    Next, we want to show that if we remove a prime surface, say $F_g=T^2$ and $F'_{g'}=T^2$, from both connected sum decompositions, the resulting surfaces are still diffeomorphic. Then we can conclude by induction. For that we choose a curve $S^1\times \{p\}\subset T^2$ that is mapped under the diffeomorphisms to $F_g$ and $F'_{g'}$ to curves $c\subset F_g$ and $c'\subset F'_{g'}$ that are disjoint from the disks in $F_g$ and $F'_{g'}$ used for the connected sums. In particular, $c$ and $c'$ also represent curves on $F$. If we cut $F$ along $c$ and glue two $2$-disks to the resulting boundary components we get $F_1\#\cdots\#F_{g-1}$. Similarly, if we cut $F$ along $c'$ and fill the boundary components with $2$-disks, we obtain $F'_1\#\cdots\#F'_{g'-1}$.
Since $c$ and $c'$ are non-separating, Lemma~\ref{lem:Dehn} gives us a diffeomorphism of $F$ that maps $c$ to $c'$. Its restriction $F\setminus c\to F\setminus c'$ extends by Lemma~\ref{lem:Alex} over the glued-in $2$-disks to a diffeomorphism from $F_1\#\cdots\#F_{g-1}$ to $F'_1\#\cdots\#F'_{g'-1}$. We conclude by induction that $g=g'$.    
\end{proof}

\begin{rem} We will close with two remarks.
    \begin{itemize}
        \item Some authors define the genus of a closed, connected, oriented surface as the maximum number of components in a non-separating curve system. Theorem~\ref{thm:unique} implies that a surface of genus $g$ (with this alternative definition) is diffeomorphic to $\#_g T^2$.
        \item Our arguments extend also to non-orientable surfaces. We briefly comment on the necessary changes. Similar as in Lemma~\ref{lem:T2prime} we can show that the real projective plane $\R P^2$ is prime. By extending Proposition~\ref{prop:prime_surfaces} we see that $\R P^2$ is the only non-orientable prime surface~\cite{Sc14}. The existence of the prime decomposition works exactly the same. (Here we only need to notice that a separating curve on a surface always has an orientable neighborhood.) 

        The uniqueness is slightly more involved. If $F$ is a non-orientable surface then it is not hard to see that $F\#T^2$ is diffeomorphic to $F\#\R P^2\#\R P^2$ and thus the prime decomposition is in general not unique. Now let $F$ be a non-orientable surface that is diffeomorphic to two connected sum decompositions, to $F_1\#\cdots\# F_g$ and to $F'_1\#\cdots\# F'_{g'}$, with $F_1,\ldots,F_g$ and $F'_1,\ldots,F'_{g'}$ prime surfaces that are not $2$-spheres. First we use the relation $F\#T^2=F\#\R P^2\#\R P^2$ to replace any $T^2$-summand by an $\R P^2\#\R P^2$-summand. Thus Proposition~\ref{prop:prime_surfaces} implies that all $F_i$ and $F'_i$ are $\R P^2$. Then we use an extension of Lemma~\ref{lem:Dehn} to non-orientable surfaces, to inductively remove $\R P^2\#\R P^2$-summands from the two different connected sum decompositions as in the proof of Theorem~\ref{thm:unique}. Like this we can reduce one of the two connected sum decompositions to $S^2$ or $\R P^2$. From the primeness and the orientability it follows that the other reduced connected sum decomposition is the same.
    \end{itemize}
\end{rem}

\section{Higher dimensional manifolds}\label{sec:higher_dimensions}

Since the proofs presented here are inspired by the $3$-dimensional proofs they naturally extend into dimension $3$, where Kneser and Milnor proved that any closed oriented $3$-manifold admits a unique (up to reordering and adding $S^3$-factors) prime decomposition~\cite{Kn29,Mi62}. On the other hand, it is not hard to show that there exist infinitely many non-homeomorphic prime $3$-manifolds (for example the lens spaces) and thus the classification of $3$-manifolds is not an implication of the prime decomposition theorem. Nevertheless, the algorithmic classification of $3$-manifolds is possible and normal surfaces play a crucial role in them as in many other algorithmic results in $3$-manifold topology. We refer to~\cite{Ha61,Ma07,BBP+} for more details. 
On the other hand, it is not hard to construct for any given finitely presented group $G$ and any integer $n\geq4$ a closed oriented $n$-manifold with fundamental group isomorphic to $G$. Since the isomorphism problem of finitely presented groups is unsolvable~\cite{No55,Bo58} there is also no algorithmic classification of manifolds of dimension at least $4$. Nevertheless, manifolds of dimension at least $5$ are classified, although not algorithmically, via surgery theory, see for example~\cite{Br72, Wa70}. 

Finally, we refer to~\cite{BC+19} for a discussion of the prime decompositions of higher dimensional manifolds. In particular, we want to emphasize the observation that the connected sum decomposition of higher dimensional smooth manifolds is often not unique and the naive approach of cutting a manifold along essential separating spheres will in general not terminate uniquely in finite time. Here we mention two different ways to see this. First, we have the Kervaire--Milnor theorem that states that there are exactly $28$ non-diffeomorphic smooth $7$-manifolds that are all homeomorphic to $S^7$ and that these $28$ manifolds together with the connected sum form a group isomorphic to $\Z_{28}$~\cite{KM63}. Second, we know that for any closed simply-connected smooth $4$-manifold $W$ there exist natural numbers $a,b,c,d$ such that $W\#_a\C P^2\#_b -\C P^2$ is diffeomorphic to $\#_c\C P^2\#_d-\C P ^2$, where $\C P^2$ denotes the complex projective space and $-\C P^2$ the complex projective space with opposite orientation~\cite{GS99}.

\end{document}